\documentclass[11pt]{article}

\usepackage[T1]{fontenc} 
\usepackage{amsmath, amssymb, amsthm, amsfonts, mathrsfs}

\newtheorem{thm}{Theorem}[section]
\newtheorem{cor}[thm]{Corollary} 
\newtheorem{lem}[thm]{Lemma}
\newtheorem{prop}[thm]{Proposition}
\newtheorem{convention}[thm]{Convention}

\newtheorem*{thm*}{Theorem}
\newtheorem*{prop*}{Proposition}
\newtheorem*{thmA}{Theorem A}

\newtheorem{defn}[thm]{Definition}
\newtheorem*{defn*}{Definition} 
\newtheorem{rem}[thm]{Remark}
\newtheorem*{rem*}{Remark}

\numberwithin{equation}{section}

\renewcommand{\phi}{\varphi}
\renewcommand{\bar}{\overline}

\newcommand{\R}{\mathbb{R}}
\newcommand{\N}{\mathbb{N}}

\newcommand{\ssm}{\smallsetminus}
\newcommand{\inv}{^{-1}}
\newcommand{\FN}{F_N}
\newcommand{\Stab}{\text{Stab}}
\newcommand{\dom}{\text{dom}}
\newcommand{\CQ}{Q}
\newcommand{\cvn}{\text{cv}_N}
\newcommand{\barcvn}{\bar{\text{cv}}_N}
\newcommand{\That}{\widehat T}
\newcommand{\Thatobs}{\widehat T^{\text{obs}}}
\newcommand{\wh}{\text{Wh}}
\newcommand{\ind}{\text{ind}}
\newcommand{\indgeom}{\text{ind}_\text{geom}}
\newcommand{\Rank}{\text{Rank}}
\newcommand{\acts}{\curvearrowright}
\newcommand{\ikl}{i_{\text{KL}}}

\begin{document}

\title{Indecomposable $\FN$-trees and minimal laminations}

\author{Thierry Coulbois, Arnaud Hilion and Patrick Reynolds}

\date{\today}

\maketitle

\begin{abstract}
  We extend the techniques of \cite{CH10} to build an inductive procedure for studying actions in the boundary of the Culler-Vogtmann
  Outer Space, the main novelty being an adaptation of the classical
  Rauzy-Veech induction for studying actions of \emph{surface type}.  As an
  application, we prove that a tree in the boundary of Outer space is
  free and indecomposable if and only if its dual lamination is
  \emph{minimal up to diagonal leaves}.  Our main result generalizes
  \cite[Proposition 1.8]{BFH97} as well as the main result of
  \cite{KL11}.
\end{abstract}

\tableofcontents

\section{Introduction}

We consider  $\R$-trees $T$ equipped with a minimal very small action 
of the free group $\FN$ of rank $N$ by isometries: these are points in the closure of the celebrated 
unprojectivized Culler-Vogtmann Outer Space $\barcvn$ \cite{CV86}.
A dual lamination $L(T)$ is associated to such trees \cite{CHL08a} . We recall that a lamination for the free group is a closed $\FN$-invariant, flip invariant, subset of the double Gromov boundary $\partial^2\FN=(\partial\FN)^2\ssm \Delta$, where $\Delta$ is the diagonal. The main result result of the paper relates \emph{minimality} properties of the tree to \emph{minimality} properties of its dual lamination.

For an action $\FN\acts T$ there are several notions of minimality for the dynamics of the action; see Section \ref{SS.Mixing}. Mixing trees were considered by Morgan \cite{Mor88}, and Guirardel introduced the stronger notion of indecomposability in \cite{Gui08}. Mixing and indecomposable trees are the elementary objects from which any tree in the closure of Outer Space can be built \cite{R10b}.

\begin{defn}
An $\R$-tree $T\in\barcvn$ is \textbf{indecomposable} if for any non-degenerate segments $I$ and
  $J$ in $T$, there exist finitely many elements $u_1,\ldots,u_n$ in
  $\FN$ such that
\begin{enumerate}
\item $I\subseteq u_1J\cup u_2J\cup\cdots\cup u_nJ$
\item $u_iJ\cap u_{i+1}J$ is a non degenerate segment for any $i=1,\ldots,n-1$.
\end{enumerate}
\end{defn}

A lamination $L$ is minimal if it does not contain a proper sublamination, or, equivalently, if the orbit of any leaf of $L$ is dense in $L$.  Laminations dual to trees are always diagonally closed: if $L(T)$ contains leaves $l_1=(X_1,X_2)$ and $l_2=(X_2,X_3)$ then $L(T)$ also contains the leaf $l=(X_1,X_3)$. Oftentimes the diagonal leaves are isolated in the dual lamination. 
Thus we need a slightly modified notion of minimality for laminations dual to $\R$-trees.

\begin{defn}
  A lamination $L$ is \textbf{minimal up to diagonal leaves} if :
  \begin{enumerate}
  \item [(i)] there is a unique minimal sublamination $L_0 \subseteq
    L$, and
  \item [(ii)] $L \ssm L_0$ consists of finitely many
    $\FN$-orbits of leaves, each of which is diagonal over $L_0$.
\end{enumerate}
\end{defn}

In the definition, a leaf $l \in L$ is \textbf{diagonal} over a
sublamination $L_0 \subseteq L$ if there are leaves $(X_1,X_2),
(X_2,X_3), \ldots (X_{n-1},X_n) \in L_0$ such that $l=(X_1,X_n)$. Note that the above notion of minimality coincides with the correct notion of minimality for foliations on a closed surface.  Our main result is:

\begin{thmA}
  Let $T$ be an $\R$-tree with a free, minimal action of $\FN$ by
  isometries with dense orbits. The tree $T$ is indecomposable if and
  only if $L(T)$ is minimal up to diagonal leaves. 
  
  In this case the unique 
  minimal sublamination is the derived sublamination $L(T)'$: the subset of non-isolated leaves.
\end{thmA}

The third author \cite{R10b} proved that under the same hypotheses, $T$ is indecomposable if and only if no leaf of the dual lamination $L(T)$ is carried by a finitely generated subgroup of infinite index. We use this characterization in the proof of Theorem~A.

Theorem~A generalizes the main result of \cite{KL11}.  Indeed, it is
shown in \cite{CH10} that the repelling tree $T_{\Phi\inv}$ of a fully irreducible (iwip) outer
automorphism $\Phi$ of $\FN$ is indecomposable, and $T_{\Phi\inv}$ is free
exactly when $\Phi$ is non-geometric. Moreover, the attracting
lamination $L_\Phi$ of $\Phi$ is minimal \cite{BFH97} and contained in the dual
lamination of the repelling tree (see  \cite{CH10}).

\begin{cor}[\cite{KL11}]
  Let $\Phi$ be a non-geometric iwip outer automorphism of $\FN$. The
  dual lamination of the repelling tree $L(T_{\Phi\inv})$ is the
  diagonal closure of the attracting lamination $L_\Phi$.\qed
\end{cor}

Recall that a current on $\FN$ is a positive $\FN$-invariant,
flip invariant, Radon measure $\mu$ on $\partial^2 \FN$ \cite{Kap06}. The support
$Supp(\mu)$ of $\mu$ is a lamination (see
\cite{CHL08c}). We consider the intersection form $\ikl$ between
currents and trees: $\ikl:\barcvn\times\text{Curr}(\FN)\to\R_{\geq 0}$ \cite{KL09a}. 
A current $\mu$ is orthogonal to a tree $T\in\barcvn$ (i.e $\ikl(T,\mu)=0$) if and only if the support of $\mu$ is a sublamination of the dual lamination $L(T)$ \cite{KL10d}. Theorem~A implies:

\begin{cor}
  Let $T\in\barcvn$ be a free, indecomposable $\FN$-tree, and let  $\mu$
  be a geodesic current. The following are equivalent:
\begin{enumerate}
\item[(i)] $Supp(\mu)\subseteq L(T)$,
\item[(ii)] $\ikl(T,\mu)=0$,
\item[(iii)] $Supp(\mu)=L(T)'$.\qed
\end{enumerate}
\end{cor}

Our strategy for proving Theorem~A is summarized as follows.  We encode a tree $T\in\barcvn$ in a system of isometries on a compact $\R$-tree as in \cite{CHL09}, and then develop an inductive procedure for studying such systems of isometries.  There are two cases to consider: if the tree $T$ is of Levitt type we use the Rips induction \cite{CH10}; if the tree $T$ is of surface type we use a splitting procedure inpired by the Rauzy-Veech induction for flat surfaces. The result is a sequence of systems of isometries encoding the tree $T$; there is a train track associated to each of these systems of isometries, and the dual lamination $L(T)$ is the inverse limit of this sequence of train tracks. We are able read off from the sequence of train tracks the desired minimality properties of the lamination.  The motivation for our procedure comes from surface
theory.

\subsection{The Case of a Surface Lamination}\label{SS.SurfaceCase}

The proof of Theorem~A relies on an adaptation to the free
group $\FN$ of \emph{train track expansions} of a measured lamination on a hyperbolic surface, which we now
casually recall.  First, we mention the correspondence between
measured laminations, interval exchanges, and train tracks, as this
correspondence suggests the procedure presented in this paper.
Let $\Sigma$ be a non-exceptional hyperbolic surface equipped with a
measured geodesic lamination $\mathscr{L}=(L,\mu)$.   We suppose that $L$ is minimal and filling. Choosing an arc $I
\subseteq \Sigma$ that is transverse to $L$, one can consider the
first return map $f:I \rightarrow I$ induced by the holonomy of $L$.
The map $f$ is a classical \emph{interval exchange transformation},
whose dynamics coincide with those of a finitely generated pseudogroup
$S$ of partial isometries on $I$: the generators of $S$ are the
maximal continuous restrictions of $f$.  Suspending $S$, see
Section~\ref{SS.Suspend}, gives a foliated space $\mathscr{S}$, which
is homeomorphic to $\Sigma$ with finitely many discs removed; after
gluing in these discs and extending the foliation to them in the
obvious way, one obtains a surface homeomorphic to $\Sigma$, carrying
a measured foliation which corresponds to measured lamination
$\mathscr{L}$; see \cite{Lev83}.
 
The suspension $\mathscr{S}$ of the interval exchange $f$ has the
homotopy type of a finite graph--collapsing each band of $\mathscr{S}$
onto one of its leaves gives the desired homotopy equivalence
$h:\mathscr{S} \rightarrow \Gamma$.  The graph $\Gamma$ can be
equipped with a \emph{train track structure} as follows: let $e$ and
$e'$ be (oriented) edges of $\Gamma$ such that the endpoint $v$ of $e$
coincides with the initial point of $e'$; make $e$ and $e'$ tangent at
the vertex $v$ if and only if there is a non-singular leaf of
$\mathscr{S}$ that crosses the (ordered) pair of bands corresponding
to $(e,e')$.  The widths of the bands of $\mathscr{S}$ give an
assignment of weights to the edges of $\Gamma$, which clearly satisfy
a \emph{switch condition}.  In fact, one sees that $\Gamma$ can be
embedded in $\Sigma$ as a measured train track \emph{carrying} $L$.

We now consider the effect of the \emph{Rauzy-Veech induction}.  Recall that
given an interval exchange transformation $f:I \rightarrow I$ with corresponding
partition $I=I_1 \cup ... \cup I_r$, one step of the Rauzy-Veech induction
consists of: first removing the right most interval $I_r$, then
replacing $f$ with the first return map $f'$ on the interval $I
\ssm I_r$.  Equivalently, consider the suspension $\mathscr{S}$,
and split apart the right-most singularity of the foliation on
$\mathscr{S}$ (all singularities lie on $I$) until $I$ is again
reached.  Rauzy-Veech induction is defined in the generic case, and extra
information is available when it is undefined; see
Proposition~\ref{prop:splittingPointExist}. This
process gives a foliated space $\mathscr{S}'$, which is the suspension
of the interval exchange transformation $f'$.  In the present note, we consider this
``splitting'' point of view of Rauzy-Veech induction.

As noted before, there is a measured train track $\Gamma$
corresponding to the space $\mathscr{S}$, such that $\Gamma$ carries
$L$.  If $\mathscr{S}'$ arises from $\mathscr{S}$ via one step of the
Rauzy-Veech induction, then we say that the measured train track $\Gamma'$
corresponding to $\mathscr{S}'$ arises from $\Gamma$ via
\emph{splitting}; there is a homotopy equivalence $\tau:\Gamma'
\rightarrow \Gamma$--$\tau$ is a \emph{fold}.  The sequence of
measured train tracks $(\Gamma_i)$ arising by iterating this procedure
is usually called a \emph{splitting sequence}, or a \emph{train track
  expansion} of $L$.  The key features for us are that each $\Gamma_i$
carries $L$ and that the branches of $\Gamma_i$ approximate leaves
of $L$ for $i >>0$.

\subsection{The Case of a Very Small $\FN$-Tree}

We now outline our translation of the above technology.  Let $T$ be an
$\R$-tree equipped with a free, minimal action of $\FN$ by isometries,
and suppose that the action $\FN \curvearrowright T$ is mixing
(see Section~\ref{SS.Mixing}).  Let $L(T)$ denote the dual
  lamination associated to $T$ (see Section~\ref{SS.Q}).  Choose a
basis $A$ for $\FN$, let $X_A$ denote the Cayley tree of $\FN$ with
respect to this basis, and identify $\partial^2 \FN = \partial^2 X_A$
in the obvious way.  Let $C_A$ denote the subset of $\partial^2 \FN$
consisting of points represented by lines in $X_A$ passing through the
identity element. The set $C_A$ is compact, hence $L_A:=L(T)
\cap C_A$ is compact as well.  As $L(T)$ is $\FN$-invariant, we may
consider the restricted action of $\FN$ on $L_A$; this turns out to be
a classical \emph{symbolic flow} (see \cite{CHL08a}).  We
``geometrize'' this flow using the tree $T$.

It is shown in \cite{CHL09} that there is an $\FN$-equivariant,
continuous map $Q^2:L(T) \rightarrow \overline{T}$, where
$\overline{T}$ denotes the metric completion of $T$; consider the
compact subspace $\Omega_A:=Q^2(L(T) \cap C_A)$--this is the
analogue of an arc transverse to a surface lamination.  From the
equivariance of $Q^2$, we get a partial action of $\FN$ on $\Omega_A$,
which is just the restriction of the action $\FN \curvearrowright T$.
It is shown in \cite{CH10} that there are two possibilites for the
structure of $\Omega_A$: either $\Omega_A$ is a finite union of
compact $\R$-trees, in which case $T$ is called \emph{surface type};
or $\Omega_A$ is totally disconnected, in which case $T$ is called
\emph{Levitt type}.  In either case the suspension of the (partial)
action of $\FN$ on $\Omega_A$ is a compact foliated space that is a
``geometric realization'' of $L_A$.  However, in the Levitt type case,
it is more convenient to work with a nicer space: put $K_A$ to be the
convex hull of $\Omega_A$ in $\overline{T}$ (so $K_A$ is a compact
$\R$-tree), and suspend the (partial) action of $\FN$ on $K_A$.  The
result is a foliated space $\mathscr{S}$ in which some leaves are 1-ended;
deleting all leaves with strictly less than two ends gives the suspension of
$\Omega_A$ (see Sections~\ref{SS.Q2}, \ref{SS.Types},
\ref{SS.Suspend}, and \ref{SS.Rips}).

We now describe our generalization of Rauzy-Veech induction; our induction
has two distinct procedures.  In the case that $T$ is surface type, we
show (Proposition~\ref{prop:splittingPointExist}) that there exist in the
suspension $\mathscr{S}$  singularities which look like
singularities which arise in foliated surface, \emph{i.e.} they are
``splittable''; splitting apart these singularities as above gives a
homotopy equivalence $h_1:\mathscr{S}' \rightarrow \mathscr{S}$ ($h_1$
``zips up'' the splitting).  There is a homotopy equivalence
$h:\mathscr{S} \rightarrow \Gamma$ to a finite graph $\Gamma$ got by
collapsing each band onto a vertical fiber, and splitting induces a
homotopy equivalence $\Gamma' \rightarrow \Gamma$, where $\Gamma'$ is
the graph associated to $\mathscr{S}'$ (see
Section~\ref{S.Splitting}).

In the case that $T$ is Levitt type, we operate on $\mathscr{S}$ using
the \emph{Rips machine}: erase from $K_A$ all points $x$ such that $x$
is an endpoint of a leaf to get $K_A' \subseteq \Omega_A$, and suspend
$K_A'$ get $\mathscr{S}'$.  Again there are homotopy equivalences $\mathscr{S}
\rightarrow \Gamma$, $\mathscr{S}' \rightarrow \Gamma'$, and $\Gamma'
\rightarrow \Gamma$.

In either case, we get a sequence of foliated spaces $(\mathscr{S}_i)$
and graphs $\Gamma_i$ with homotopy equivalences $\tau_i:\Gamma_i
\rightarrow \Gamma_{i-1}$.  Any sublamination $L_0 \subseteq L(T)$
defines a \emph{train track structure} on $\Gamma_i$ just as in
Section~\ref{SS.SurfaceCase}, where leaves in $L_0$ are treated as
``non-singular.''  As one might expect, choosing too large of a
sublamination $L_0$ gives too many legal turns; however, it follows
from the results of this paper that for $T$ \emph{indecomposable} (see
Section~\ref{SS.Mixing}), the derived set $L(T)'$ consists of leaves
that are morally non-singular; see Sections~\ref{SS.Derived} and
\ref{SS.TrainTracks}.  As before, the main features are that the train tracks 
$\Gamma_i$ all carry $L(T)$ and that for $i>>0$, the branches
of $\Gamma_i$ ``approximate'' leaves of $L(T)'$.  Our train-tracks will have 
no transverse measure; however, the
informed reader will realize that currents are the natural
object for measuring these train tracks.  This idea will be explored
in future work.

\medskip

\noindent{\bf Acknowledgments:} This project began during the AIM
workshop ``The geometry of the outer automorphism group of a free
group,'' which was held in October 2010; we thank the
organizers of that conference as well as the American Institute of
Mathematics.

\section{Background}

In this section we briefly review the relevant definitions around
$\mathbb{R}$-trees, Outer space, and  laminations.  In what
follows $\FN$ denotes the free group of rank $N$.

\subsection{Basics About $\mathbb{R}$-Trees}

A metric space $(T,d)$ is an $\mathbb{R}$-\textbf{tree} (or just a
\textbf{tree}) if for any two points $x,y \in T$, there is a unique
topological arc $p_{x,y}:[0,1] \rightarrow T$ connecting $x$ to $y$,
and the image of $p_{x,y}$ is isometric to the segement $[0,d(x,y)]$.
As is usual, we let $[x,y]$ stand for Im$(p_{x,y})$, and we call
$[x,y]$ the \textbf{segment} (also called an \textbf{arc}) in $T$ from
$x$ to $y$.  A segment is \textbf{non-degenerate} if it contains
strictly more than one point.  We let $\overline{T}$ stand for the
metric completion of $T$.  Unless otherwise stated, we regard $T$ as a
topological space with the metric topology.  If $T$ is a tree, and $x
\in T$, then $x$ is a \textbf{branch point} if the cardinality of
$\pi_0(T - \{x\})$ is strictly greater than two.  For $x \in T$, the
elements of $\pi_0(T- \{x\})$ are the \textbf{directions} at $x$.

In this paper, all the trees we consider are equipped with an
isometric (left) action of $\FN$, i.e. a group morphism $\rho: \FN
\rightarrow$ Isom$(T)$; as usual, we supress the morphism
$\rho$ and identify $\FN$ with $\rho(\FN)$.  A tree $T$ equipped with
an isometric action will be called an $\FN$-\textbf{tree}, and we
denote this situation by $\FN \curvearrowright T$.  Notice that an
action $\FN \curvearrowright T$ induces an action of $\FN$ on the sets
of directions and branch points of $T$.  We identify two $\FN$-trees
$T,T'$ if there is an $\FN$-equivariant isometry between them.

There are two sorts of isometries of trees: an isometry $g$ of $T$ is
\textbf{elliptic} if $g$ fixes some point of $T$, while an isometry
$h$ of $T$ is \textbf{hyperbolic} if it is not elliptic.  Any
hyperbolic isometry $h$ of $T$ leaves invariant a unique isometric
copy of $\mathbb{R}$ in $T$ which is the \textbf{axis}, $A(h)$, of
$h$.  If $g$ is an elliptic isometry, we let $A(g)$ stand for the
fixed point set of $g$, \emph{i.e.} $A(g):=\{x \in T|gx=x\}$.  The
\textbf{translation length function} of a $\FN$-tree $T$ is $l_T:\FN
\rightarrow \mathbb{R}$, where
\[
l_T(g):=\inf \{d(x,gx)|x \in T\}.
\]
The number $l_T(g)$ is the \textbf{translation length} of $g$, and for
any $g \in \FN$, the infimum is always realized on $A(g)$, so that $g$
acts on $A(g)$ as a translation of length $l_T(g)$.  If $H \leq \FN$
is a finitely generated subgroup containing a hyperbolic isometry,
then $H$ leaves invariant the set
\[
T_H^{min}:=\cup_{l_T(h)>0} A(h).
\]
which is a subtree of $T$, and is minimal in the set of $H$-invariant
subtrees of $T$; $T_H^{min}$ is the \textbf{minimal invariant subtree
  for} $H$.  In the case that $H=\FN$, we omit $H$ and write $T^{min}$
for $T_{\FN}^{min}$.  An action $\FN \curvearrowright T$ is
\textbf{minimal} if $T=T^{min}$; a minimal action $\FN
\curvearrowright T$ is \textbf{non-trivial} if $T$ contains strictly
more than one point.

\subsection{System of isometries}\label{SS.Systems}

A metric space $F$ is a \textbf{finite forest} if $F$ has finitely
many connected components, each of which is a compact $\R$-tree.  A
\textbf{partial isometry} $a$ of $F$ is an isometry between two closed
(hence compact) subtrees of $F$.

A \textbf{system of isometries} is a pair $S=(F,A)$ where $F$ is a
finite forest and $A$ is a finite collection of non-empty partial
isometries of $F$. By allowing inverses and compostion we get a 
pseudo-group of partial isometries on $F$ which we write on the right. 

To a system of isometries $S=(F,A)$ we associate a graph $\Gamma$
whose vertices are the connected components of $F$ and such that for
each partial isometry $a\in A$ there is an oriented edge starting at
the connected component of $F$ containing the domain, $\dom (a)$, and
ending at the connected component of $F$ containing the image of $a$.
Denote by $V(\Gamma)$ and $E(\Gamma)$ the sets of vertices and edges
of $\Gamma$, respectively.

By a \textbf{path} in $\Gamma$, we mean a finite \textbf{edge path},
that is, a path starting and ending at vertices of $\Gamma$.  A path
$\gamma$ in $\Gamma$ is called \textbf{reduced} if $\gamma$ is an
immersion; any path is homotopic relative to endpoints to a reduced
path.  A reduced path $\gamma$ in $\Gamma$ defines a partial isometry
of $F$, and we say that $\gamma$ is \textbf{admissible} if the
corresponding partial isometry has non-empty domain. We abuse
notation, identifying an admissible path $\gamma$ with the partial
isometry corresponding to it.  An \textbf{infinite reduced path}
$\gamma$ is an immersion $\gamma:\mathbb{R}_{\geq 0}\rightarrow
\Gamma$ such that $\gamma^{-1}(V(\Gamma))=\mathbb{N}$.  For an
infinite reduced path $\gamma$, the \textbf{$i$-prefix} of $\gamma$,
denoted $\gamma_i$, is the restriction $\gamma|_{[0,i]}$.  For $i \geq
j$, one has that $\dom(\gamma_i)\subseteq\dom(\gamma_j)$, and we put
$\dom(\gamma) :=\cap_i \dom(\gamma_i)$. We say that an infinite path
$\gamma$ is \textbf{admissible} if $\dom(\gamma) \neq \emptyset$. A
\textbf{bi-infinite reduced path} $\gamma$ is an immersion
$\gamma:\mathbb{R}\rightarrow \Gamma$ such that
$\gamma^{-1}(V(\Gamma))=\mathbb{Z}$.  Given a bi-infinite path
$\gamma$, the \textbf{halves} of $\gamma$ are the restrictions
$\gamma^-:=\gamma|_{\mathbb{R}_{\leq 0}}$ and
$\gamma^+:=\gamma|_{\mathbb{R}_{\geq 0}}$; we reparametrize $\gamma^-$
in order to regard it as an infinite path.  A bi-infinite reduced path
is \textbf{admissible} if the domains of its halves have non-empty
intersection.  Equivalently a bi-infinite path $\gamma$ is admissible
if and only if for any $i \leq j \in \mathbb{Z}$, the restriction
$\gamma|_{[i,j]}$ is admissible.

The set of bi-infinite admissible paths is called the
\textbf{admissible lamination} of the system of isometries $S=(F,A)$
and is denoted $L(S)$.

\subsection{The Suspension and the Dual Tree}\label{SS.Suspend}

Let $S=(F,A)$ be a system of isometries, and let $I=[0,1]$ denote the unit
interval.  For each $a_i \in A$, let $b_i:=\dom(a_i)$; one forms a
\textbf{band} $B_i:=b_i \times I$.  Identify $b_i$ with $b_i \times
\{0\} \subseteq B_i$, and denote $\tilde{b_i}:=b_i \times \{1\}$.  Say
that $b_i$ and $\tilde{b_i}$ are the \textbf{bases} of the band $B_i$.

\begin{defn}
  The \textbf{suspension} $\mathscr{S}$ of $S$ is the quotient of $F
  \amalg \coprod_i B_i$, where one identifies $b_i$ with $\dom(a_i)$
  and $\tilde{b_i}$ with $\text{im}(a_i)=\dom(a_i)\cdot a_i$.
\end{defn}

The suspension of a system of isometries is a compact, Hausdorff space
that has the homotopy type of a finite graph $\Gamma$. The graph
$\Gamma$ is a deformation retract of $\mathscr{S}$ obtained by
contracting each connected component of the forest $F$ to a point and
each band $b_i \times I$ to a core $\{pt\}\times I$.  The graph
$\Gamma$ is the graph associated to the system of isometries in
Section~\ref{SS.Systems}.  In the sequel, we always suppose that
$\mathscr{S}$ is connected.  Hence, $\pi_1(\mathscr{S})$ is a free
group $\FN$.

We think of each band as being \textbf{foliated} by leaves of the form
$\{pt\} \times [0,1]$, and the foliations of the bands of
$\mathscr{S}$ give rise to a foliation on $\mathscr{S}$: define a
relation $R$ on points of $\mathscr{S}$ by declaring $xRy$ if and only
if $x$ and $y$ lie in the same leaf of the foliation on some band.
The classes of the smallest equivalence relation containing $R$ are
the \textbf{leaves} of the foliation on $\mathscr{S}$.  For a point $x
\in F$, we let $l(x)$ denote the leaf of the foliation on
$\mathscr{S}$ containing $x$.

We consider the path metric on leaves coming from the metric on
$[0,1]$.  A finite, infinite, or bi-infinite path $\gamma$ in
$\mathscr{S}$ is an \textbf{admissible leaf path} if $\gamma:
J\rightarrow \mathscr{S}$, $J$ a closed subinterval of $\R$ with
extremities $\partial J\subseteq \mathbb{Z}\cup\{\pm\infty\}$, is a
locally isometric, immersed leaf path, so $\gamma^{-1}(F) = J\cap
\mathbb{Z}$.  The \textbf{lamination} (which is rather a foliation in
this setting) $L(\mathscr{S})$ is the set of bi-infinite admissible
leaf paths.

Any admissible leaf path $\gamma$ defines an admissible path in
$\Gamma$ which we also denote by $\gamma$. Any bi-infinite admissble
leaf path in $L(\mathscr{S})$ defines a bi-infinite admissible path in
the admissible lamination $L(S)$.  Thus, the admissible lamination
$L(S)$ associated to a system of isometries $S=(F,A)$ is a
combinatorial version of the foliation on the suspension
$\mathscr{S}$.

A \textbf{length measure} $\mu$ on a tree $T$ is a collection
$\{\mu_I\}_{I\subseteq T}$ of finite Borel measures on the compact
arcs $I$ of $T$, such that if $J \subseteq I$, then $\mu_J=\mu_I|_J$.
For any tree $T$, we have the \textbf{Lebesgue measure} $\mu_L$ on $T$
consisting of the Lebesgue measures on the compact intervals of $T$.
The foliated space $\mathscr{S}$ inherits a \textbf{transverse measure}
from the Lebesgue measures on the bases. 

Let $S=(F,A)$ be a system of isometries, and let $\tilde{\mathscr{S}}$
denote the universal cover of $\mathscr{S}$.  Lift the foliation and
transverse measure from $\mathscr{S}$ to $\tilde{\mathscr{S}}$.  Then
$\FN=\pi_1(\mathscr{S})$ acts on $\tilde{\mathscr{S}}$ by deck
transformations, and the foliation and transverse measure are
preserved.  Collapsing to a point each leaf of the foliation of
$\tilde{\mathscr{S}}$ gives an $\R$-tree $T_{\mathscr{S}}$.  As the
action $\FN \curvearrowright \tilde{\mathscr{S}}$ preserves the
foliation and transverse measure, one gets an isometric action $\FN
\curvearrowright T_{\mathscr{S}}$; see \cite{CH10} or \cite{BF95}.
The action $\FN \curvearrowright T_{\mathscr{S}}$ is \textbf{dual} to
the system of isometries $S$.

\subsection{Outer Space and its Closure}

Recall that an action $\FN \curvearrowright T$ is \textbf{free} if for
any $1 \neq g \in \FN$ one has $l_T(g) > 0$.  If $X \subseteq T$, then
the \textbf{stabilizer} of $X$ is $\Stab(X):=\{g \in \FN|gX=X\}$  ---  the
setwise stabilizer of $X$.  An action $\FN
\curvearrowright T$ is \textbf{very small} if:

\begin{enumerate}
 \item [(i)] $\FN \curvearrowright T$ is minimal,
 \item [(ii)] for any non-degenerate arc $I \subseteq T$, $\Stab(I) =
   \{1\}$ or $\Stab(I)$ is a maximal cyclic subgroup of $\FN$,
 \item [(iii)] stabilizers of tripods are trivial.
\end{enumerate}
 
A minimal action $\FN \curvearrowright T$ is \textbf{discrete}
(or \textbf{simplicial}) if the $\FN$-orbit of any point of $T$ is a
discrete subset of $T$; in this case $T$ is obtained by equivariantly
assigning a metric to the edges of a (genuine) simplicial tree. Note
that the metric topology is weaker than the simplicial topology if the
tree is not locally compact.

The \textbf{unprojectivised Outer Space} of rank $N$, denoted $\cvn$,
is the topological space whose underlying set consists of free,
minimal, discrete, isometric actions of $\FN$ on $\mathbb{R}$-trees. A
minimal $\FN$-tree is completely determined by its translation length
function (see, for example, \cite{Chi}): we can embed $\cvn
\subseteq\mathbb{R}^{\FN}$. The closure $\barcvn$ in $\R^{\FN}$ consists
of very small isometric actions of $\FN$ on $\mathbb{R}$-trees
\cite{CL95, BF94}.  For more background on $\cvn$ and its closure, see
\cite{Vog02} and the references therein.

\subsection{The Map $Q$ and the Dual Lamination}\label{SS.Q}

Here we recall dual algebraic laminations associated to $\FN$-trees;
see \cite{CHL08a} and \cite{CHL08b} for a careful development of the
general theory.  Let $\partial \FN$ denote the Gromov boundary of
$\FN$ --- \emph{i.e.} the Gromov boundary of any Cayley graph of $\FN$;
let $\partial^2(\FN):=\partial \FN \times \partial \FN\ssm\Delta$,
where $\Delta$ is the diagonal.  The left action of $\FN$ on a Cayley
graph induces actions by homeomorphisms of $\FN$ on $\partial \FN$ and
$\partial^2 \FN$.  Let $i: \partial^2 \FN \rightarrow \partial^2 \FN$
denote the involution that exchanges the factors.  A
\textbf{lamination} is a non-empty, closed, $\FN$-invariant,
$i$-invariant subset $L\subseteq \partial^2 \FN$.

\begin{rem}
  In the setting of Sections~\ref{SS.Systems} and \ref{SS.Suspend}, if
  the graph $\Gamma$ is connected, then its fundamental group is a free
  group $\FN$.  Specifying a marking isomorphism
  $\pi_1(\Gamma)\simeq\FN$ gives a homeomorphism
  $\partial\tilde\Gamma\simeq\partial\FN$, where $\tilde\Gamma$ is the
  universal cover of $\Gamma$. A bi-infinite reduced path in $\Gamma$
  lifts to bi-infinite reduced paths in $\tilde\Gamma$ that are
  completly described by their pairs of ends in $\partial^2\FN$. We
  regard $L(S)$ and $L(\mathscr{S})$ as laminations.
\end{rem}

\begin{prop}[\cite{LL03}]\label{P.DefQ}
  Let $T \in \barcvn$ have dense orbits, and suppose that $X
  \in \partial \FN$. There is a unique point $Q(X) \in
  \overline{T}\cup\partial T$ such that there exists a sequence $u_n$
  in $\FN$ converging to $X$ and a point $P\in T$ such that $u_nP$
  converges to $\CQ(X)$.\qed
\end{prop}

More intuitively, the map $\CQ$ is the continuous extension of the map
$\CQ_P:\FN\to T$, $u\to uP$ to $\partial\FN\to\Thatobs$, where
$\That=\bar T\cup\partial T$ is endowed with the weaker observers'
topology--the set of directions in $\That$ is a basis of open sets for the
observers' topology. The space $\Thatobs$ is Hausdorff and compact.

\begin{prop}[\cite{LL03}]\label{prop:Qonto}
  Let $T \in \barcvn$ have dense orbits. The map
  $\CQ:\partial\FN\to\That$ is $\FN$-equivariant and surjective; further, points in
  $\partial T$ have exactly one pre-image by $\CQ$.\qed
\end{prop}

The crucial property for us is that $Q$ can be used to
associate to $T$ a lamination \cite{CHL08b}.

\begin{defn}\label{D.DefL2}
  Let $T \in \barcvn$ have dense orbits.  The \textbf{dual
    lamination} of $T$ is
  \[
  L(T):=\{(X,Y) \in \partial^2(\FN)|Q(X)=Q(Y)\}.
  \]
\end{defn}

\subsection{The Map $Q^2$ and the Compact Heart}\label{SS.Q2}

For a tree $T\in\barcvn$ with dense orbits, the map $Q:\partial \FN
\rightarrow \That=\bar T\cup\partial T$ induces a map $Q^2:L(T)
\rightarrow \overline{T}:(X,Y) \mapsto Q(X)=Q(Y)$.  The metric
topology on $T$ canonically extends to $\overline{T}$, and we have:

\begin{prop}[\cite{CHL09}]\label{P.Q2Cont}
 The map $Q^2: L(T) \rightarrow \overline{T}$ is continuous.\qed
\end{prop}

The space $\partial^2 \FN$ is not compact, but there are many ``nice'' coverings of $\partial^2 \FN$ be compact sets. Fixing a basis $A$ for
$\FN$ gives an identification of $\partial \FN$ with the space of
infinite reduced words in $A^{\pm 1}$, hence an identification of
$\partial^2 \FN$ with the space of pairs $(X,Y)$ of distinct infinite
reduced words $X\neq Y$.  For an infinite word $X$, we let $X_1$ stand
for the first letter of $X$.  The \textbf{unit cylinder} of
$\partial^2 \FN$ with respect to $A$ is the subset $C_A:=\{(X,Y)
\in \partial^2 \FN | X_1 \neq Y_1\}$.  For any basis $A$ of $\FN$,
$C_A$ is a compact subspace of $\partial^2 \FN$, and $\partial^2 \FN =
\cup_{g \in \FN} gC_A$.

For a tree $T \in \barcvn$ with dense orbits the \textbf{limit
  set} of $T$ is $\Omega=\CQ^2(L(T))\subseteq\bar T$. For any basis
$A$ of $\FN$, the \textbf{compact limit set} of $T$ (with respect to
$A$) is $\Omega_A:=Q^2(L(T) \cap C_A)\subseteq \overline{T}$.  The
\textbf{heart} of $T$ (with respect to $A$) is the convex hull $K_A$
of $\Omega_A$ in $\overline{T}$.  It follows from
Proposition~\ref{P.Q2Cont} that $\Omega_A$ is a compact subset of
$\overline{T}$, so $K_A$ is a compact subtree of $\overline{T}$.

Now, fix a tree $T \in\barcvn$ with dense orbits and a basis $A$ for
$\FN$.  For a compact subtree $K \subseteq \overline{T}$, let
$S=(K,A)$ denote the system of isometries with $a \in A$ the partial
isometry got by the (maximal) restriction of the action of $a\inv$ to
$K$.  Let $\mathscr{S}$ denote the suspension of $S$, and let
$T_{\mathscr{S}}$ denote the $\FN$-tree dual to $\mathscr{S}$.  The
following result is essential for the present note.

\begin{prop}[\cite{CHL09}]\label{P.HeartProperties}
  Let $T \in \barcvn$ have dense orbits, and fix a basis $A$
  for $\FN$.  Let $K_A$ denote the heart of $T$ with respect to $A$;
  let $S=(K_A, A)$ denote the associated system of isometries, and let
  $\mathscr{S}$ denote its suspension.
\begin{enumerate}
\item [(i)] The tree $T$ is dual to $S$: $T_\mathscr{S}^{\text{min}}=T$, 
\item [(ii)] $L(T) \cap C_A=L(S)=L(\mathscr{S})$, and
\item[(iii)] for any infinite admissible leaf path $X$ in
  $\mathscr{S}$, $\dom(X)=\{\CQ(X)\}$.\qed
\end{enumerate}
\end{prop}

Proposition~\ref{P.HeartProperties} allows us to fully transfer the
problem of understanding the dual lamination of a tree in the boundary
of Outer space to the problem of understanding the associated systems
of isometries.

A tree $T\in\barcvn$ is \textbf{geometric} if its compact heart $K_A$
for some (hence any) basis $A$ of $\FN$ is a finite tree, that is to
say a compact $\R$-tree which is the convex hull of finitely many
points, see \cite{CHL09}.

\subsection{Regular Leaves and the Derived Sublamination}\label{SS.Derived}

In this Section, we collect some results regarding the derived
space of $L(T)$ and its relationship to regular leaves in
systems of isometries associated to $T$.

\begin{defn}
  Let $T$ be an $\FN$-tree in $\barcvn$ with dense orbits. A leaf
  $l\in L(T)$ is \textbf{regular} if there exists a sequence $l_n\in
  L(T)$ of leaves converging to $l$ and such that the $x_n=\CQ^2(l_n)$
  are distinct. The set of regular leaves is the \textbf{regular
    sublamination} $L_r(T)$.
\end{defn}

If the action of $\FN$ on $T$ is free, by \cite[Theorem 5.3]{CH10} we
have that for all $x\in\bar T$, $(Q^2)^{-1}(x)$ is a finite set of
uniformly bounded cardinality. Recall that the set of non-isolated
points of a topological space, $X$, is its \textbf{derived space}, $X'$.

\begin{lem}\label{C.LdBasisInd}
  Let $T \in \barcvn$ be free with dense orbits. The regular
  sublamination of $T$ is equal to the derived lamination:
  \[
  L_r(T)=L(T)'.\qed
  \]
\end{lem}

\subsection{Mixing Properties for $\FN$-trees}\label{SS.Mixing}

From the work of V.~Guirardel \cite{Gui08} (after
J.~Morgan \cite{Mor88}), there are several notions of
``minimality'' for the dynamics of an action of
a group on an $\R$-tree. These notions are hierarchized as
follows:
\begin{enumerate}
\item \textbf{dense orbits}: the $\FN$-orbit of some (hence any)
  point $P$ of $T$ is dense in $T$;
\item \textbf{arc-dense}: every orbit meets every non-degenerate segment of $T$;
\item \textbf{arc-dense directions}: for each $x \in T$, each
  direction $d$ at $x$, and each non-degenerate arc $I \subseteq T$,
  there exists $g\in\FN$ such that $gx\in I$ and $gd\cap I$ is
  non-degenerate;
\item \textbf{mixing}: for any non-degenerate segments $I$ and $J$ in
  $T$, there exists finitely many elements $u_1,\ldots,u_n$ in $\FN$
  such that $I\subseteq u_1J\cup u_2J\cup\cdots\cup u_nJ$;
\item \textbf{indecomposable}: for any non-degenerate segments $I$ and
  $J$ in $T$, there exists finitely many elements $u_1,\ldots,u_n$ in
  $\FN$ such that
\begin{enumerate}
\item $I\subseteq u_1J\cup u_2J\cup\cdots\cup u_nJ$
\item $u_iJ\cap u_{i+1}J$ is a non degenerate segment for any $i=1,\ldots,n-1$.
\end{enumerate}
\end{enumerate}

We remark that this hierarchy is not exactly strict as

\begin{lem}[{\cite[Lemma 12.6]{R10b}}]\label{L.MixArcDense}
  Let $T \in \barcvn$.  The action $\FN\curvearrowright T$ is
  mixing if and only if it has arc-dense directions.\qed
\end{lem}

In this paper we will use two characterizations of indecomposable
trees.  A \textbf{transverse family} for an action $\FN
\curvearrowright T$ of $\FN$ on an $\R$-tree $T$ is an $\FN$-invariant
family $\{T_v\}_{v\in V}$ of non-degenerate, proper subtrees of $T$
such that if $T_v \neq T_{v'}$, then $T_v\cap T_{v'}$ contains at most
one point.

\begin{prop}[{\cite[Lemma 4.1]{R10a}}]\label{P.IndNoTF}
  Let $\FN \curvearrowright T$ be an action of $\FN$ on an
  $\mathbb{R}$-tree $T$. Then $\FN \curvearrowright T$ is
  indecomposable if and only if there is no transverse family.\qed
\end{prop}

In the proof of Lemma~\ref{L.NearlyMin}, we need a refined understanding of transverse families that occur in free $\FN$-trees. We collect the following:

\begin{prop}[{\cite[Lemma~4.4]{R10b}} and
  {\cite[Theorem~5]{Lev94}}] \label{prop:transverse} Let $T\in\barcvn$
  be free with dense orbits. If $T$ is not indecomposable there exists
  a non-degenerate subtree $T_0$ of $T$, such that 
  \begin{enumerate}
  \item $\{gT_0\ |\ g\in\FN\}$ is a transverse family,
  \item $H=\Stab(T_0)$ is a free factor of $\FN$,
\item $H$ acts on $T_0$ with dense orbits.\qed
\end{enumerate}
\end{prop}

A finitely generated subgroup $H$ of $\FN$ is quasiconvex, thus the
boundaries $\partial H\subseteq\partial\FN$ come with a natural
inclusion. We say that a line $(X,Y)\in\partial^2\FN$ is
\textbf{carried} by $H$ if $(X,Y)\in\partial^2H$.

\begin{prop}[{\cite[Corollary~4.8]{R10a}}]\label{P.Indecomp}
  Let $T \in \barcvn$ be free and indecomposable, and let $H
  \leq \FN$ be finitely generated.  Then $H$ carries a leaf of
  $L(T)$ if and only if $H$ has finite index in $\FN$.\qed
\end{prop}

\section{The Rips Machine and Types of Actions}

\subsection{The Rips Machine}\label{SS.Rips}

We recall the generalization of Process I of the Rips Machine
\cite{GLP94,BF95} that was first studied in the present contex in
\cite{CH10}.

Let $S=(F,A)$ be a system of isometries.  The output of one step of
the Rips Machine applied to $S$ is a new system of isometries $S'=(F',
A')$ defined as follows:
\[
 F':=\{x \in F| \exists a \neq a' \in A^{\pm 1}, x \in \dom(a)\cap\dom(a')\}
\]
Since $A$ is finite and since intersections of domains of isometries
are compact $\mathbb{R}$-trees, we have that $F'$ is again a finite
forest.  We let $A'$ consist of all maximal restrictions of the
elements of $A$ to pairs of connected components of $F'$, so
$S'=(F',A')$ is indeed a system of isometries, as required.

The suspension $\mathscr{S}'$ of $S'$ is a subspace of the suspension
$\mathscr{S}$ of $S$. We can regard each leaf-path in $\mathscr{S}'$
as a leaf-path in $\mathscr{S}$, in particular for bi-infinite
admissible leaf paths $L(\mathscr{S}')\subseteq L(\mathscr{S})$.
On the other hand, the Rips Machine does not modify bi-infinite admissible
leaf paths, thus:

\begin{lem}\label{L.NonSingRips}
  Let $S=(F,A)$ be a system of isometries, and let $S'=(F',A')$ denote
  the output of the Rips machine applied to $S$.  The laminations are
  equal: $L(\mathscr{S})=L(\mathscr{S}')$.\qed
\end{lem}

\subsection{Types of Actions}\label{SS.Types}

We consider the output of iterating the Rips Machine on a system
of isometries $S_0=(F_0, A_0)$; we denote by $S_i$ the output of the
$i^{th}$ iteration of the Rips Machine.  If for some $i_0$, one has
that $F_{i_0}=F_{i_0+1}$, \emph{i.e.} the Rips Machine \textbf{halts} on
$S_{i_0}$, then we say that the Rips Machine \textbf{eventually halts}
on $S_0$.

\begin{defn}
  Let $S_0$ be a system of isometries.  If the Rips Machine eventually
  halts on $S_0$, then $S_0$ is called \textbf{surface type}.
\end{defn}

The \textbf{limit set} of $S_0$ is $\Omega=\cap_{i\in\N}F_i$. If $S_0$
is of surface type then $\Omega=F_{i_0}$. If $S=(K_A,A)$ is a system
of isometries associated to a tree $T\in\barcvn$ with dense orbits,
then the limit set $\Omega$ of the system of isometries is equal to
the compact limit set $\Omega_A$ with respect to the basis $A$ defined
in Secton~\ref{SS.Q2}.

\begin{defn}
  Let $S_0$ be a system of isometries, and suppose that the Rips
  machine does not eventually halt on $S_0$.  If the limit set $\Omega$
  associated to $S_0$ is totally disconnected, then $S_0$ is said to
  be \textbf{Levitt type}.
\end{defn}

In \cite{CH10} it is shown that for $T \in \barcvn$ with dense
orbits, if for some basis $A$, the system of isometries associated to
$K_A$ is of surface type (resp. Levitt type), then for every basis
$A'$, the system of isometries associated to $K_{A'}$ is of surface
type (resp. Levitt type).  In this case we say that $T$ is of
\textbf{surface type} (resp. \textbf{Levitt type}).  It should be noted
that there are trees in $\barcvn$ that are neither of surface
type nor Levitt type; however, we have the following:

\begin{prop}[{\cite[Proposition 5.14]{CH10}}]\label{P.MixType}
  Let $T \in \barcvn$ have dense orbits.  If the action $\FN
  \curvearrowright T$ is mixing, then $T$ is either of surface type or
  Levitt type.\qed
\end{prop}

\subsection{Levitt type Actions}

Let $T \in \barcvn$ has dense orbits, and let $S_0=(K_A,
A)=(F_0,A_0)$ be an associated system of isometries.  Denote by $S_i$
the output of the $i^{th}$ iteration of the Rips machine.   

Recall the definition of the graph $\Gamma$ associated to a system of
isometries $S$ from Section~\ref{SS.Systems}.  Let $\Gamma_i$ denote
the graph associated to $S_i$: $\Gamma_i$ is got by contracting each
band of $\mathscr{S}_i$ onto one of its leaves. There are
induced graph morphisms $\tau_i:\Gamma_i \rightarrow \Gamma_{i-1}$.
The following Lemma follows from \cite[Propositions~3.12, 3.13, and 5.6]{CH10}.

\begin{lem}\label{L.TauIsHE}
  Let $T \in \barcvn$ have dense orbits; let $A$ a basis for
  $\FN$; and let $S_0$ denote the associated system of isometries.
  Denote by $S_i$ the output of the $i^{th}$ iteration of the Rips
  machine applied to $S_0$, and let $\Gamma_i$ be the associated graph.
\begin{enumerate}
 \item [(i)] $\Gamma_i$ has no vertices of valence 0 or 1, and 
 \item [(ii)] the maps $\tau_i :\Gamma_i \rightarrow \Gamma_{i-1}$ are
   homotopy equivalences.\qed
\end{enumerate}  
\end{lem}

Note that, as $F_0=K_A$ is connected, $\Gamma_0$ is a rose with $N$
petals, so Lemma~\ref{L.TauIsHE} gives a uniform bound $2N-2$ on the
number of vertices of valence strictly greater than two in $\Gamma_i$.

\begin{lem}\label{L.Lev}
  Let $T \in \barcvn$ be free with dense orbits, and suppose
  that $T$ is of Levitt type.  If $L_0 \subsetneq L_r(T)$ is a
  proper sublamination, then every leaf of $L_0$ is carried by a
  proper free factor of $\FN$.
\end{lem}
\begin{proof}
  Fix a basis $A$ for $\FN$, and let $S_0=(K_A, A)$ be the associated
  system of isometries; let $S_i=(F_i, A_i)$ denote the output of the
  $i^{th}$ iteration of the Rips machine, and let $\Gamma_i$ denote
  the graph associated to $S_i$.  Recall that since $T$ is of Levitt
  type, the limit set $\Omega$ is totally disconnected, hence the
  number of vertices of $\Gamma_i$ goes to infinity with $i$.  Let $l$
  be a bi-infinite admissible leaf path in $L_r(\mathscr{S}_0)\ssm
  L_0$. There exists a sequence $l_n$ in $L(\mathscr{S}_0)$ converging
  to $l$ such that the $x_n=\CQ^2(l_n)$ are distinct and distinct from
  $x=\CQ^2(l)$.  Additionally, we can assume that
  $l_n|_{[-n,n]}=l|_{[-n,n]}$ (viewed as admissible paths in
  $\Gamma_0$).  As ${\Omega}$ is totally disconnected, for any $m$,
  there is $i(m)$ such that $x_n$ lie in separate components of $F_i$
  for $n \leq m$ and $i \geq i(m)$.  Also, since $l \notin L_0$, there
  is $M$ such that
  \[
  \{l' \in L_0\ |\ l'|_{[-M,M]}=l|_{[-M,M]}\}=\emptyset.
  \]
  We now apply Lemma~\ref{L.NonSingRips} to view the leaves $l_n$ as
  bi-infinite admissible leaf-paths in the suspensions
  $\mathscr{S}_i$.

  For a given $m$ the leaves $l_n$, $n\leq m$, define distinct
  admissible paths $l_m|_{[-M;M]}$ in $\Gamma_{i(m)}$ each of length
  $2M$.  As the action of $\FN$ on $T$ is free there exists $j(M)$
  such that for $i\geq j(M)$ the size of any reduced loop in
  $\Gamma_i$ is strictly bigger than $2M$. Recall that there are at
  most $2N-2$ vertices of valence strictly bigger than two in
  $\Gamma_i$. Thus for $m$ large enough, for any $i$ bigger than
  $i(m)$ and $j(M)$, there exists $n\leq m$ such that the admissible
  reduced path $l_n|_{[-M;M]}$ in $\Gamma_i$ does not cross any vertex
  of valence strictly greater than two.

  Note that no leaf of $L_0$ could cross any edge in the image of
  $l_n|_{[-M,M]}$ in $\Gamma_{i}$; thus every leaf of $L_0$ is
  contained the the subgraph $G_0:=(\Gamma_{i}\ssm
  Im(l_n|_{[-M,M]})$.  By Lemma~\ref{L.TauIsHE}, we have that $G_0$
  corresponds to a proper free factor $H \leq \FN$ and every leaf of
  $L_0$ is carried by $H$.
\end{proof}

We now have our first main  result.

\begin{prop}\label{P.LevMin}
  Let $T \in \barcvn$, and assume that $T$ is free and
  indecomposable and of Levitt type.  The regular sublamination
  $L_r(T)\subseteq L(T)$ is minimal.
\end{prop}
\begin{proof}
  According to Lemma~\ref{L.Lev}, if there happened to be a proper
  sublamination $L_0 \subsetneq L_r(T)$, then every leaf of $L_0$
  would be carried by a proper free factor $H$ of $\FN$.  This is
  impossible by Proposition~\ref{P.Indecomp}.
\end{proof}

\section{Splitting}\label{S.Splitting}

In this section we define an inductive procedure that allows us to
study the dual lamination of a free, surface type tree.  To define
this procedure, which is a generalization of the classical Rauzy-Veech
induction, we need to find ``good'' singularities in systems of
isometries associated to trees of surface type.  Toward that end we
recall some results regarding indices of trees.

\subsection{$Q$-Index and Geometric Index}

Let $T \in \barcvn$ have dense orbits, and let $Q: \partial
\FN \rightarrow \That=\bar T\cup \partial T$ be the map defined in
Section~\ref{SS.Q}.  For $x \in \overline{T}$, let $\Stab(x) \leq \FN$
denote the stabilizer of $x$.  It was shown in \cite{GL95} that there
are finitely many orbits of points in $T$ with non-trivial stabilizer
and that $\Stab(x)$ is finitely generated.  Note that for $x
\in \That\ssm T$, $\Stab(x)$ is always trivial.  From the definition
of $Q$, one sees that $\partial\Stab(x) \subseteq Q^{-1}(x)$; put
$Q^{-1}_r(x):=Q^{-1}(x) \ssm \partial\Stab(x)$.  Evidently, $\Stab(x)$
acts on $Q^{-1}(x)$, leaving invariant $Q^{-1}_r(x)$.  For $x \in
\That$, the \textbf{$Q$-index} of $x$ is
\[
 \ind_Q(x):=|Q^{-1}_r(x)/\Stab(x)| + 2 \Rank(\Stab(x)) -2
\]
The $Q$-index is constant on $\FN$-orbits in $T$, and the
\textbf{$Q$-index of $T$} is
\[
 \ind_Q(T):=\sum_{[x] \in \That/\FN} \max \{0,\ind_Q(x)\}
\]

As $Q$ is injective on $Q^{-1}(\partial T)$, only points of $\bar T$
contribute to the $Q$-index of $T$.  The following is established in
\cite{CH10}:

\begin{thm}[{\cite[Theorem 5.3]{CH10}}]\label{T.QIndex}
  Let $T \in \barcvn$ have dense orbits.  Then $\ind_Q(T) \leq
  2N-2$. Moreover, $T$ is surface type if and only if
  $\ind_{Q}(T)=2N-2$.\qed
\end{thm}

Let $T \in \barcvn$ have dense orbits, and let $x \in T$.
Then $\Stab(x)$ acts on $\pi_0(T\ssm\{x\})$.  Following
\cite{GL95}, one defines the \textbf{geometric index} of $x$ to be
\[
 \indgeom(x):=|\pi_0(T\ssm \{x\})/\Stab(x)| + 2\Rank(\Stab(x)) -2
\]

The geometric index is constant on $\FN$-orbits in $T$, and
one defines the \textbf{geometric index} of $T$ to be
\[
 \indgeom(T):=\sum_{[x] \in T/\FN} \indgeom(x)
\]

\noindent We have the following:

\begin{thm}[\cite{GL95}]\label{T.GeomIndex}
  Let $T \in \barcvn$.  Then $\indgeom(T) \leq 2N-2$. Moreover,
  $T$ is geometric if and only if $\indgeom(T)=2N-2$.\qed
\end{thm}

\subsection{Finding Splitting Points}

\begin{convention}
  If $T \in \barcvn$ is of surface type and if $B$ is a basis
  for $\FN$, we let $S=(K_B,B)$ denote the associated system of
  isometries.  In the sequel, we assume that that any system of
  isometries $S=(F,A)$ associated to a surface type action is obtained
  from some $S=(K_B,B)$ by running the Rips machine until it halts.
\end{convention}

A point $x$ is \textbf{extremal} in a tree $K$ if it is not contained
in the interior of an arc contained in $K$.  In this section we are
interested in points which are extremal in some bases of a system of
isometries but which are non-extremal in the underlying forest.

Let $S=(F,A)$ be a system of isometries. A partial isometry $a\in
A^{\pm 1}$ is \textbf{defined in direction} $d$ at a point $x\in F$
if $d$ is a direction in $\dom(a)$: $x\in\dom(a)$ and
$d\cap\dom(a)\neq\emptyset$.

\begin{prop}\cite[Proposition 4.3]{CH10}\label{prop:surfaceDirections}
  Let $T\in\barcvn$ be a tree with dense orbits of surface type and
  $S=(F,A)$ be a system of isometries associated to $T$.  For each
  direction $d$ at a point $x$ in $F$ there are exactly two partial
  isometries $a,b\in A^{\pm 1}$ defined in $d$.\qed
\end{prop}

In the surface type case, we can locally compare the geometric and $\CQ$-indices.

\begin{lem}\label{L.BoundQFiber}
  Let $T \in \barcvn$ be free with dense orbits of surface
  type, and let $S=(F,A)$ be a system of isometries associated to $T$.
  Let $x \in F$ lie in the intersection of at least three distinct
  bases.  If every point of the $S$-orbit of $x$ is non-extremal in
  all bases that contain it, then $\ind_\CQ(x)\leq\indgeom(x)$.
\end{lem}
\begin{proof}
  Let $\Gamma_x$ the (infinite) graph with vertices $V(\Gamma_x)$ the
  $S$-orbit of $x$ and with an edge labeled by $a\in A^{\pm 1}$
  between each pair of vertices $x.u$ and $x.ua$ with $u\in\FN$ a
  partial isometry defined at $x$.  As $T$ is free, the graph
  $\Gamma_x$ is a tree and from Proposition~\ref{P.HeartProperties},
  its space of ends can be identified with
  $\CQ\inv(x)\subseteq\partial\FN$. From Theorem~\ref{T.QIndex},
  $\Gamma_x$ has finitely many ends.

  Let $\Gamma_x^d$ be the (infinite) graph with vertices the
  directions in $F$ at points in the $S$-orbit of $x$ and with an edge
  labeled by $a\in A^{\pm 1}$ between each pair of vertices $d$ and
  $d.a$ (in particular $a$ is defined in $d$). The number of connected
  components in $\Gamma_x^d$ is $\indgeom(x)+2$. From
  Theorem~\ref{T.GeomIndex}, $\Gamma^d_x$ has finitely many connected
  components.

  By Proposition~\ref{prop:surfaceDirections}, $\Gamma_x^d$ is a
  disjoint union of bi-infinite lines. By our hypothesis on $x$, for
  each edge labeled by $a$ from $x.u$ to $x.ua$ in $\Gamma_x$ there
  are at least two edges labeled by $a$ in $\Gamma_x^d$ from $d_1$ to
  $d_1.a$ and from $d_2$ to $d_2.a$ where $d_1,d_2\in V(\Gamma_x^d)$
  are directions at $x.u$. 

  Each end of $\Gamma_x$ is reached by at least to bi-infinite lines
  in $\Gamma_x^d$ and a bi-infinite line has two ends, thus the number
  of lines in $\Gamma_x^d$ is bounded below by the number of ends in
  $\Gamma_x$:
  \[
  \ind_\CQ(x)\leq\indgeom(x).\qedhere
  \]
\end{proof}

\begin{defn}\label{def:splitting}
  A \textbf{splitting point} in a system of isometries $S=(F,A)$ is a
  point $x$ in the connected component $K_x$ in $F$ such that
  \begin{enumerate}
    \renewcommand{\theenumi}{(S\arabic{enumi})}
    \renewcommand{\labelenumi}{\theenumi}
  \item\label{defSplit:1} $x$ is not extremal in $K_x$;
  \item\label{defSplit:2} there exists a partial isometry $a_0\in A^{\pm 1}$ defined at
    $x$ such that $x$ is extremal in the base $\dom(a_0)$, and $\dom(a_0)$
    is not reduced to $\{x\}$. We denote by $d_x$ the unique
    direction at $x$ which meets $\dom(a_0)$. We call $d_x$ the
    \textbf{splitting direction}
  \item\label{defSplit:3} There exists exactly one other partial isometry $a_1\in A^{\pm
      1}\ssm\{a_0\}$ defined at $x$ and such that $\dom(a_1)$ meets
    $d_x$.
  \item\label{defSplit:4} The point $x$ is not extremal in $\dom(a_1)$.
  \end{enumerate}
\end{defn}  

\noindent From the previous Proposition we get that splitting points exists.

\begin{prop}\label{prop:splittingPointExist}
  Let $T\in\barcvn$ be free, indecomposable, and of surface type. Let $S=(F,A)$ be a system
  of isometries dual to $T$. There exists a splitting point $x$ in
  $S$.
\end{prop}
\begin{proof}
  We first prove that there exist points satisfying
  conditions~\ref{defSplit:1} and \ref{defSplit:2}. 

  If there were no points satisfying conditions~\ref{defSplit:1} and
  \ref{defSplit:2} then we can use Proposition~\ref{L.BoundQFiber} to get  
  \[
  \indgeom(T)=\sum_{[x]\in T/\FN} \indgeom(x)\geq \sum_{[x]\in T/\FN}
  \ind_\CQ(x)=\ind_\CQ(T).
  \]
  As $T$ is of surface type, from Theorem~\ref{T.QIndex},
  $\ind_\CQ(T)=2N-2$ and by Theorem~\ref{T.GeomIndex} we get that $T$
  is geometric. By definition of geometric trees the forest $F$ has
  finitely many extremal points and as conditions~\ref{defSplit:1} and
  \ref{defSplit:2} fail, partial isometries send extremal points to
  extremal points, and the action is not free, a contradiction.

  Thus, there exists a point $x$ in $F$ and a partial isometry $a_0\in
  A^{\pm 1}$ satisfying conditions~\ref{defSplit:1} and
  \ref{defSplit:2}. As $T$ is of surface type, according to
  Proposition~\ref{prop:surfaceDirections} condition~\ref{defSplit:3}
  is satisfied. If condition~\ref{defSplit:4} does not hold then $x$
  locally separates the suspension $\mathscr{S}$ of $S$.  In this case there is a proper free factor $F'$ of $F_N$ that carries every leaf of $L_r(\mathscr{S})$, contradicting Proposition \ref{P.Indecomp}.
\end{proof}

\subsection{Splitting}

Let $S=(F,A)$ be a system of isometries, and let $\Gamma$ be its associated graph. Assume that $x$ is a
splitting point for $S$. We use the notations of 
Definition~\ref{def:splitting}.

We split the connected component $K_x$ of $x$ into two new compact
$\R$-trees: $K'=d_{x}\cup\{x\}$ and $K''=K_x\ssm d_{x}$.  We denote by
$F'$ the finite forest obtained by replacing $K_x$ by two disjoint
compact $\R$-trees $K'$ and $K''$ (in particular there are two copies
of $x$ in $F'$). Let $A_0$ be the set of partial isometries which do
not meet $x$:
\[
A_0=\{ a\in A\ |\ x\not\in \dom(a)\cup\dom(a\inv)\}.
\]
Let $a_0'$ be the same partial isometry as $a_0$ with domain in $K'$.
Let $a_1'$ be the restriction of $a_1$ to $K'$ and $a_1''$ be the
restriction of $a_1$ to $K''$.  For any other partial isometry $a\in
A^{\pm 1}\ssm\{a_0,a_1\}$ such that $x\in\dom(a)$ we let $a''$ be the
same partial isometry as $a$ defined on $K''$.  Put
$A'=A_0\cup\{a_0',a_1',a_1''\}\cup\{a''\ |\ a\in A^{\pm
  1}\ssm\{a_0,a_1\},\ x\in\dom(a)\}$.  We say that the system of isometries $S'=(F',A')$ is obtained from
$S=(F,A)$ by \textbf{splitting} at $x$ in the splitting direction $d_x$.

The suspension $\mathscr{S}'$ of $S'$ can be
``zipped-up'' to recover the suspension $\mathscr{S}$ of $S$: the map
$z:\mathscr{S}'\to\mathscr{S}$ which identify the leaves $\{(x,t)\ |\ t\in [0;1]\}$ in
the band $K'\times [0;1]$ with $\{(x,t)\ |\ t\in [0;1]\}$ in the band
$K''\times [0,1]$ is a homotopy equivalence.

\begin{lem}\label{lem:zippingLeaves}
  A regular bi-infinite admissible leaf $l$ in $\mathscr{S}$ can be
  lifted by $z$ to a regular bi-infinite admissible leaf in
  $\mathscr{S}'$.
\end{lem}
\begin{proof}
  There exists a sequence $l_n$ of bi-infinite admissible leaves in
  $\mathscr{S}$ converging to $l$ such that the $x_n=\CQ(l_n)$ are
  distinct and distinct from $x$. We can assume that for each $n$ the
  finite admissible leaf path $l_n|_{[-n,n]}$ does not cross $x$. From the definition of splitting, $l_n|_{[-n,n]}$ can be lifted to a finite
  admissible leaf path $\gamma_n$ in $\mathscr{S}'$. The paths
  $\gamma_n$ converges to a bi-infinite admissible leaf path $l'$ in
  $\mathscr{S}'$ which is a lift of $l$.
\end{proof}

We can now iteratively split our system of isometries dual to a tree
of surface type to get the analogue of Lemma~\ref{L.TauIsHE}. As there
might be more than one splitting point (there are at most finitely many), we
choose, as a convention, to split all splitting points simultaneously.

\begin{lem}\label{L.TauIsHE2}
  Let $T \in \barcvn$ be free with dense orbits, and suppose
  that $T$ is of surface type.  Let $S_0$ denote a system of
  isometries associated to $T$.  Denote by $S_i$ the output of
  splitting $S_{i-1}$, and let $\Gamma_i$ be the graph associated to $S_i$.
\begin{enumerate}
 \item [(i)] $\Gamma_i$ has no vertices of valence 0 or 1, and
 \item [(ii)] the maps $\tau_i :\Gamma_i \rightarrow \Gamma_{i-1}$ are
   homotopy equivalences.\qed
\end{enumerate}  
\end{lem}

\subsection{Surface Type Actions}

We now establish analogues of Lemma~\ref{L.Lev} and
Proposition~\ref{P.LevMin} for actions of surface type.

\begin{lem}\label{L.Surface}
  Let $T \in \barcvn$ be free with dense orbits, and suppose
  that $T$ is of surface type.  If $L_0 \subsetneq L_r(T)$
  is a proper sublamination, then every leaf of $L_0$ is carried by a
  proper free factor of $\FN$.
\end{lem}
\begin{proof}
  Let $S=(K,A)$ be a system of isometries associated to $T$. By
  definition of surface type after finitely many steps the Rips
  machine starting on $S$ halts on a surface type system of isometries
  $S_0$.  According to Proposition~\ref{prop:splittingPointExist} and its proof and
  Corollary~\ref{L.TauIsHE2}, either we can then perform splittings on
  $S_0$ or every leaf of $L_r(T)$ is carried by a proper free factor of $\FN$. Let $S_i=(F_i, A_i)$ denote the result of the $i^{th}$
  iteration of splitting applied to $S_0$. 

  By Lemma~\ref{L.MixArcDense}, directions are arc dense in $T$ and by
  Proposition~\ref{P.HeartProperties}, directions are arc dense in $F_i$
  under the action of the pseudogroup $S_i$.  In particular if $d_x$ is
  the first splitting direction, for any non-degenerate arc $[y,y']
  \subseteq F$, there is a finite admissible path $\gamma$ in the
  graph $\Gamma_0$ associated to $S_0$ such that $x.\gamma_0\in
  [y,y']$ and $d_x.\gamma_0$ meets $[y,y']$.  For all $i \geq
  i(y,y')=|\gamma|$, the images of $y$ and $y'$ in $F_i$ lie in
  different components of $F_i$.

  Using Lemma~\ref{lem:zippingLeaves} and Propostion~\ref{L.TauIsHE2},
  we may now conclude exactly as in the proof of Lemma~\ref{L.Lev}.
\end{proof}

\noindent We have:

\begin{prop}\label{P.SurfaceMin}
  Let $T \in \barcvn$, and assume that $T$ is free,
  indecomposable and of surface type.  The regular sublamination
  $L_r(T)\subseteq L(T)$ is minimal.
\end{prop}
\begin{proof}
  According to Lemma~\ref{L.Surface}, if there happened to be a proper
  sublamination $L_0 \subsetneq L_r(T)$, then every leaf of $L_0$ would
  be carried by a proper free factor $H\leq \FN$.  This is
  impossible by Proposition~\ref{P.Indecomp}.
\end{proof}

\section{Diagonal Leaves}

Recall from the introduction that a lamination $L$ is minimal up to
diagonal leaves if $L$ contains a unique minimal sublamination $L_0$, such
that $L \ssm L_0$ consists of finitely many $\FN$-orbits of leaves
that are diagonal over $L_0$.

In this section we use both the Rips machine and the splitting
induction to reach our main result.

\subsection{Decomposable trees}

\begin{lem}\label{L.NearlyMin}
  Let $T\in\barcvn$ be free with dense orbits.  If $L(T)$ is
  minimal up to diagonal leaves, then $T$ is indecomposable.
\end{lem}
\begin{proof}
  We argue the contrapositive.  Let $T\in\barcvn$ have dense orbits,
  and suppose that $T$ is not indecomposable.  Following
  Proposition~\ref{prop:transverse}, there exists a non-degenerate transverse family
  $\{gT_0\ |\ g\in\FN\}$, where $T_0$ is a closed non-degenerate
  subtree of $T$. The stabilizer $H=\Stab(T_0)$ is a proper free
  factor of $\FN$ and acts on $T_0$ with dense orbits.  The dual
  lamination $L_H(T_0)\subseteq\partial^2H$ of $T_0$ is non-empty, and
  $L_H(T_0)$ is diagonally closed. Recall that $H$ is quasi-convex in
  $\FN$ and thus there is an embedding
  $\partial^2H\subseteq\partial^2\FN$. Moreover, as $H$ is a free
  factor we have:
  \[
  \forall g\in\FN, g\partial H\cap\partial H\neq\emptyset\iff g\in H.
  \]
  The lamination generated by $L(T_0)$, $L_0=\FN.L(T_0)$, is a
  sublamination of $L(T)$ closed by diagonal leaves. 

  Assume first that there exists $g\in\FN\ssm H$ such that $g\bar
  T_0\cap \bar T_0\neq\emptyset$, and let $x,y\in\bar T_0$ such that
  $x=gy$. By Proposition~\ref{prop:Qonto} the map $\CQ:\partial
  H\to\That_0$ is onto, thus there exist, $X,Y\in\partial H$ such that
  $\CQ(X)=x$ and $\CQ(Y)=y$. By definition of the dual lamination,
  $(X,gY)$ is a leaf of $L(T)$; by construction $(X,gY)$ is not diagonal over $L_0$.

  We now assume, that $\forall g\in\FN\ssm H$, $g\bar T_0\cap
  T_0=\emptyset$. The action of $\FN$ on $T$ has dense orbits, thus
  there exists a sequence $g_n\in \FN\ssm H$ such that
  $d(T_0,g_nT_0)<\frac 1n$. Fixing a basis $B=\{a_1,\ldots a_r\}$ of
  $H$ which is completed to a basis $A=\{a_1,\ldots, a_N\}$ of $\FN$,
  we can write $g_n=h_n\cdot g'_n$ in reduced form with $h_n\in H$ and
  $g'_n$ starting with a letter in $A\ssm B$. Of course
  $d(T_0,g'_nT_0)<\frac 1n$. The action of $H$ on $T_0$ has dense
  orbits thus for any point $y\in T_0$ there exist $h'_n\in H$ such that
  $d(T_0,g'_nh'_ny)<\frac 1n$. By our assumption the sequence
  $|g'_nh'_n|$ goes to infinity and there is a subsequence converging to
  $Y\in \partial\FN$. The first letter of $Y$ written as an infinite
  reduced word is in $A\ssm B$ thus $Y\not\in\partial H$. Now we use
  the weaker observers' topology so that $\Thatobs$ is compact and, we
  extract again a subsequence to have $g'_nh'_ny$ converging to a
  point $x\in\Thatobs_0$. We get that $\CQ(Y)=x$, but as $x\in\hat
  T_0$ there exists $X\in\partial H$ such that $\CQ(X)=x$. By
  definition of the dual lamination, $(X,Y)$ is a leaf in $L(T)$; by construction $(X,Y)$
  is not diagonal over $L_0$.
\end{proof}

Considering Propositions~\ref{P.LevMin} and \ref{P.SurfaceMin} and
Lemma \ref{L.NearlyMin}, to establish Theorem~A, we need
understand diagonal leaves in $L(T)$ for $T$ free and indecomposable.

\subsection{Train Tracks and the Main Result}\label{SS.TrainTracks}

Let $T$ be a free, indecomposable tree in $\barcvn$.  Let $A$ be a
basis for $\FN$, and let $S=(K_A, A)$ be the associated system of
isometries.  By Proposition~\ref{P.MixType}, $T$ is either suface or
Levitt type.  If $T$ is surface, we run the Rips Machine on $S$ until it
halts.  In either case, we get a system of isometries $S_0=(F_0,A_0)$
(which is equal to $S$ if $T$ is Levitt type), and we denote by
$S_i=(F_i, A_i)$ the result of running either the Rips machine or
splitting on $S_0$ for $i$ steps.  There are homotopy equivalences
$S_i \rightarrow \Gamma_i$ and $\tau_i:\Gamma_i \rightarrow
\Gamma_{i-1}$. 

A \textbf{turn} in $\Gamma_i$ is a pair $\{e,e'\}$ of directed edges
with the same initial vertex.  We give the graph $\Gamma_i$ a
\textbf{train track structure} by declaring a turn \textbf{legal} if
it is crossed by a regular leaf, \emph{i.e.} a regular leaf contains the subpath $\overline{e}e'$.  Train track structures on graphs
were introduced in \cite{BH92}.

\begin{rem}\label{R.InductionFolds}
  From Propostions~\ref{L.TauIsHE} and \ref{lem:zippingLeaves}, our
  inductive procedure (either the Rips machine or splitting) applied
  to $\mathscr{S}_i$ has the effect of ``splitting'' an illegal turn
  in $\Gamma_i$: in other words, the graph morphisms $\tau_i$ only
  fold at illegal turns.
\end{rem}

For a vertex $v$ of $\Gamma_i$, let $Leg(v)$ denote the set of legal
turns in $\Gamma_i$ at $v$, and let $I(v)$ denote the set of edges of
$\Gamma_i$ with initial vertex $v$.  Following \cite{BH92} again, we
define the \textbf{Whitehead graph}, $\wh(v,\Gamma_i)$, associated to
the vertex $v$ of $\Gamma_i$. The vertex set of $\wh(v,\Gamma_i)$ is
$I(v)$ and there is an edge from $e$ to $e'$ if the turn $\{e,e'\}$ is
legal.

\begin{lem}\label{L.Whitehead}
  For every $v \in V(\Gamma_i)$, the Whitehead graph $\wh(v,\Gamma_i)$
  is connected.
\end{lem}
\begin{proof}
  Toward contradiction suppose that there is $i$ and $v \in
  V(\Gamma_i)$ such that $\wh(v,\Gamma_i)$ is not connected.
  Following the proof of \cite[Proposition 4.5]{BH92} this proves that
  every regular leaf of $L(T)$ is carried by a proper free factor of
  $\FN$, contradicting Proposition~\ref{P.Indecomp}.
\end{proof}

\begin{lem}\label{L.2EndDer}
  For any $x$ in the limit set $\Omega$, there exists a regular leaf
  $l\in L(T)$ such that $\CQ^2(l)=x$. In particular, if $l\in L(T)$ is
  such that $(\CQ^2)\inv(\CQ^2(l))=\{l\}$ then $l$ is regular.
\end{lem}
\begin{proof}
  We first translate $x$ in the compact limit set $\Omega_A$. For each
  $i$ let $v_i$ be the connected component of $F_i$ containing $x$. By
  Lemma~\ref{L.Whitehead}, there is a bi-infinite regular admissible
  leaf-path $l_i$ passing through $v_i$. Up to passing to a
  subsequence, $l_i$ converge to a bi-infinite regular admissible leaf path
  $l$. By the continuity of $\CQ^2$ and arguing as in the proof of
  Lemma~\ref{L.Surface} and \ref{L.Lev} we get that $\CQ^2(l)=x$.
\end{proof}

\begin{prop}\label{P.Diag}
  Let $T \in \barcvn$ be free and indecomposable.  Every leaf in $L(T)
  \ssm L_r(T)$ is diagonal over $L_r(T)$, and there are finitely
  many $\FN$-orbits of such leaves.
\end{prop}
\begin{proof}
  Let $l$ be a leaf in $L(T)$, and let $\CQ^2(l)=x$. If
  $(\CQ^2)\inv(x)=\{l\}$, then by Lemma~\ref{L.2EndDer}, $l$ is
  regular. Assume now that $|(\CQ^2)\inv(x)|>1$. From \cite{CH10},
  there are finitely many orbits of such points $x$ in $\bar T$ and as
  the action is free, $(\CQ^2)\inv(x)$ is finite and there are
  finitely many orbits of such leaves $l$.  By Lemma~\ref{L.2EndDer}, there are regular leaves in
  $(\CQ^2)\inv(x)$, and we now proceed to prove that $l$ is in the
  diagonal closure of the regular leaves in $(\CQ^2)\inv(x)$. 
 
  Let $A$ be a basis of $\FN$ and let $S=(K_A,A)=(F_0,A_0)$ be the
  system of isometries associated to $T$, and let $S_i=(F_i, A_i)$
  denote the output of $i$ iterations of the appropriate inductive
  procedure (either the Rips Machine or splitting, depending on the
  type of $T$).  Let $\Gamma_i$ denote the graph associated to $S_i$.

  Let $\Gamma_x$ be the (infinite graph) with vertex set the
  pseudo-orbit of $x$ under the pseudo-group $S$ (equivalently it is
  the intersection of the orbit of $x$ in $\bar T$ with $F_0=K_A$), and
  such that there is an edge labeled by $a\in A^{\pm 1}$ between $x.u$
  and $x.ua$, where $u\in\FN$ is a partial isometry defined at $x$.

  A turn $\{e,e'\}$ at the vetex $y=x.u$ in $\Gamma_x$ is legal if there
  exists a regular bi-infinite leaf-path $l'\in (\CQ^2)\inv(y)$ such
  that $l'([0,1])=e$ and $l'([0,-1])=e'$. As we may have performed
  splittings, the point $y$ may lie in more than one connected
  component of $F_i$. The Whitehead graph $\wh(y,\Gamma_x)$ at $y$ is
  the image of the union of the Whitehead graphs $\wh(v,\Gamma_i)$ for
  all components $v$ of $F_i$ which contain a copy of $y$. From
  Lemma~\ref{L.Whitehead}, we get that $\wh(y,\Gamma_x)$ is connected.  It follows that $l$ is in the diagonal closure of the set
  of regular leaves in $(\CQ^2)\inv(x)$.
\end{proof}

Combining Propositions~\ref{P.LevMin}, \ref{P.SurfaceMin}, and
\ref{P.Diag} and Lemma~\ref{L.NearlyMin} and \ref{C.LdBasisInd}, we get our main result:

\begin{thmA}
  Let $T$ be an $\R$-tree with a free, minimal action of $\FN$ by
  isometries with dense orbits. The tree $T$ is indecomposable if and
  only if $L(T)$ is minimal up to diagonal leaves. In this case the unique 
  minimal sublamination of $L(T)$ is the regular sublamination, which is equal to derived sublamination of $L(T)$.\qed
\end{thmA}

\bibliographystyle{alpha}
\bibliography{indecompREF}

\end{document}